\documentclass[journal,twoside,letterpaper]{ieeeconf}

\usepackage{amssymb}
\usepackage{bm}
\usepackage{amsmath}
\usepackage{graphicx}
\usepackage{xcolor}
\usepackage{hyperref}
 \hypersetup{colorlinks=true, linkcolor=blue, breaklinks=true, urlcolor=blue, citecolor=blue}
\usepackage{mathtools}

\newtheorem{theorem}{Theorem}

\newtheorem{lemma}[theorem]{Lemma}

\newcommand{\A}{{\cal A}}

\newcommand{\cohes}{\Delta_{\Graph}}
\newcommand{\cons}{{2,\Pin}}
\newcommand{\dcc}{d_{\rm cc}}
\newcommand{\dg}{{\rm deg}}
\newcommand{\dmax}{d_{\max}}
\newcommand{\Dug}{\Gamma({\bm{u},\gamma})}

\newcommand{\Graph}{{\cal G}}

\newcommand{\hev}{f^{\rm e}}
\newcommand{\hod}{f^{\rm o}}
\newcommand{\htot}{f}

\newcommand{\Htot}{F}
\newcommand{\Jac}[1]{D\!{#1}}
\newcommand{\JacNoSpace}[1]{D{#1}}

\newcommand{\one}{{\rm span}(\bm{1}_n)}

\newcommand{\Pin}{\Pi_n}
\newcommand{\proj}{{\rm proj}_{\bm u}}
\newcommand{\Pu}{P_{\bm u}}
\newcommand{\Pug}{P_{\bm{u},\gamma}}

\newcommand{\xs}{{\bm x}^{\rm s}}
\newcommand{\xt}{{\bm x}^{\rm t}}

\usepackage{xparse}
\NewDocumentCommand\seminorm{mg}{\vert\kern-0.25ex\vert\kern-0.25ex\vert #1 \vert\kern-0.25ex\vert\kern-0.25ex\vert \IfNoValueF{#2}{_{#2}}}
\NewDocumentCommand\leftseminorm{mg}{\left\vert\kern-0.25ex\left\vert\kern-0.25ex\left\vert #1 \right\vert\kern-0.25ex\right\vert\kern-0.25ex\right\vert \IfNoValueF{#2}{_{#2}}}

\begin{document}

\title{Semicontraction and Synchronization of Kuramoto-Sakaguchi Oscillator Networks}
\author{Robin Delabays and Francesco Bullo, \IEEEmembership{Fellow, IEEE}
\thanks{RD was supported by the Swiss National Science Foundation, under grant number P400P2\_194359. 
FB was supported in part by AFOSR grant FA9550-22-1-0059}
\thanks{R. Delabays is with the Institute of Sustainable Energy, School of Engineering, University of Applied Sciences and Arts of Western Switzerland (e-mail: robin.delabays@hevs.ch).}
\thanks{F. Bullo is with the Center for Control, Dynamical Systems, and Computation, University of California at Santa Barbara, Santa Barbara, CA 93106, USA (e-mail: bullo@ucsb.edu).}
\thanks{The authors thank Saber Jafarpour for insightful conversations.}
}

\maketitle
\thispagestyle{empty}

\begin{abstract}
 This paper studies the celebrated Kuramoto-Sakaguchi model of coupled oscillators adopting two recent concepts. 
 First, we consider appropriately-defined subsets of the $n$-torus called winding cells. 
 Second, we analyze the semicontractivity of the model, i.e., the property that the distance between trajectories decreases when measured according to a seminorm.  
 This paper establishes the local semicontractivity of the Kuramoto-Sakaguchi model, which is equivalent to the local contractivity for the reduced model. 
 The reduced model is defined modulo the rotational symmetry. 
 The domains where the system is semicontracting are convex phase-cohesive subsets of winding cells. 
 Our sufficient conditions and estimates of the semicontracting domains are less conservative and more explicit than in previous works. 
 Based on semicontraction on phase-cohesive subsets, we establish the \emph{at most uniqueness} of synchronous states within these domains, thereby characterizing the multistability of this model.
\end{abstract}

\maketitle

\section{Introduction}
When a group of dynamical agents interact and adapt their behavior to their neighbors' state, collective dynamics can emerge. 
\emph{Synchronization} is one such phenomenon, where all agents start following identical trajectories. 
Collective dynamics in general, and synchronization in particular, have been a major topic of research for centuries~\cite{Str04,Huy93,Win67,Kur84} and still are an active field of research. 

Among the effort in the scientific analysis of synchronization, a
breakthrough occurred through the work of Kuramoto~\cite{Kur84}, building
on the formalism of Winfree~\cite{Win67}.  Shortly after, Sakaguchi,
Shinomoto, and Kuramoto introduced the \emph{Kuramoto-Sakaguchi
model}~\cite{Sak88}, covering a larger class of dynamical systems.  The
impact of these models on the research in dynamical system is still vivid
nowadays~\cite{Are08} and the multistability phenomenon is widely
studied~\cite{Lin19,Bro18,Bal19}.  Over the last decades, various analyses of
phase oscillator networks have been performed, leveraging a diversity of
mathematical tools.  Among these tools, \emph{contraction theory} remains
under-exploited, mostly due to the technical burdens to its applicability
to oscillator networks~\cite{Wan05} and a lack of knowledge about refined
notions of contraction.

As discussed originally in~\cite{WL-JJES:98}, strongly contracting
dynamical systems have numerous properties (e.g., global exponential
stability of a unique equilibrium point and the existence of Lyapunov
functions~\cite{SC:19}), are finding increasingly widespread applications
(e.g., in controls~\cite{HT-SJC-JJES:21} and learning~\cite{Bof21}), and
their study is receiving increasing
attention~\cite{HT-SJC-JJES:21,ZA-EDS:14b,Bul23}.  However, numerous
example systems fail to satisfy this strong property and exhibit richer
dynamics; examples include systems with symmetries or conserved quantities.
In an attempt to tackle wider classes of systems, variations of contraction
were developed over the years, such as transverse
contraction~\cite{IRM-JJES:14}, horizontal contraction~\cite{FF-RS:14}, or
$k$-contraction~\cite{CW-IK-MM:22}.  Starting with the work on
partial~\cite{Wan05} and horizontal~\cite{FF-RS:14} contraction, a theory
of semicontractivity, i.e., contractivity with respect to seminorms, is now
available~\cite{Jaf22b,DeP23}.

In this manuscript, we apply semicontraction theory to the
Kuramoto-Sakaguchi model on complex networks, and we show that
semicontraction naturally allows us to deal with the symmetries of such
systems.  Our first contribution (Thm~\ref{thm:max_cohes}) is to
provide an accurate estimate (Fig.~\ref{fig:heatmap}) of the regions of the
state space where networks of Kuramoto-Sakaguchi oscillators are
semicontracting.  Theorem~\ref{thm:max_cohes} improves similar
approximations presented in Refs.~\cite{Jaf22,Del22}, in that it is more
general, less conservative, and the estimate we obtain clearly emphasizes
the role of network structures and parameters on the system's behavior
(Fig.~\ref{fig:bounds}).  Furthermore, our proof is significantly more
streamlined than the ones in \cite{Jaf22,Del22} and generalizable to a
wider class of systems.

As an application of the semicontractivity results in
Theorem~\ref{thm:max_cohes}, we then identify regions of the state spaces
containing at most a unique synchronous state (Theorem~\ref{thm:amu}),
recovering the main results of Refs.~\cite{Jaf22,Del22}.  While previous
analyses of existence and uniqueness of synchronous states were very {\it
  ad hoc} \cite{Jaf22,Del22}, our proofs are generalizable and shed some
light on the fundamental reasons leading to \emph{at most uniqueness}, i.e.,
the local semicontractivity of the dynamics.

The key analytical and computational advantage of semicontraction theory is
that it allows one to perform computations on the original vector field,
without the need to obtain and then manipulate an explicit expression for
the reduced dynamics. This advantage leads to our very concise proof of
semicontractivity (see the proofs of Lemma~\ref{lem:jaco},
Lemma~\ref{lem:jace} and Theorem~\ref{thm:max_cohes}).

In the following, we first review the Kuramoto-Sakaguchi model (Sec.~\ref{sec:ks-all}) as well as semicontraction theory (Sec.~\ref{sec:semic-recap}). 
Our main semicontractivity and \emph{at most uniqueness} results are presented in Secs.~\ref{sec:semicontraction} and \ref{sec:amu} respectively.

\section{The Kuramoto-Sakaguchi model on complex graphs}\label{sec:ks-all}
Let $\Graph$ denote a weighted, undirected, connected graph with $n$ vertices
and $m$ edges.  The matrices $B\in\mathbb{R}^{n\times m}$,
$\A\in\mathbb{R}^{m\times m}$, and $L_{\Graph}=B\A B^\top$ are the
\emph{incidence}, \emph{weight}, and \emph{Laplacian} matrices of $\Graph$
respectively.  The Laplacian's eigenvalues are $0 = \lambda_1(L_{\Graph}) <
\lambda_2(L_{\Graph}) \leq \ldots \leq \lambda_n(L_{\Graph})$, and $\lambda_2(L)$
is known as the \emph{algebraic connectivity} of $\Graph$~\cite{Fie73}.
$\bm{1}_n\in\mathbb{R}^n$ is the vector with all components equal to one
and $I_n\in\mathbb{R}^{n\times n}$ is the identity matrix.

A \emph{simple path} in $\Graph$ is a sequence of vertices
$p=(i_1,...,i_\ell)$, such that each consecutive pair is connected in $\Graph$
and each vertex appears at most once in $\sigma$.  A \emph{cycle} of $\Graph$
is a simple path whose ends ($i_1$ and $i_\ell$) are connected, and we
denote it as $\sigma=(i_0,i_1,...,i_\ell=i_0)$.

\subsection{\bf The Kuramoto-Sakaguchi model}\label{sec:ks}
We consider the \emph{Kuramoto-Sakaguchi model}~\cite{Sak88} over a weighted undirected graph $\Graph$ with frustration parameter $\varphi\in[0,\pi/2]$,
\begin{align}\label{eq:ks}
 \dot{x}_i &= \htot_i(\bm{x}) \coloneqq \omega_i - \sum_{j=1}^n a_{ij}\left[\sin(x_i - x_j - \varphi) + \sin\varphi\right],
\end{align}
where $\omega_i\in\mathbb{R}$ is the \emph{natural frequency} of agent $i$ and $a_{ij}$ is the weight of the edge between $i$ and $j$.

Due to the rotational invariance of Eq.~\eqref{eq:ks}, the
Kuramoto-Sakaguchi model is often considered as evolving on the $n$-torus.
In our case, however, it will be more convenient to look at it as evolving
in the Euclidean space $\mathbb{R}^n$.  The diffusive nature of the
couplings in Eq.~\eqref{eq:ks} implies that adding a constant value to each
degree of freedom leaves the dynamics invariant.  The Kuramoto-Sakaguchi
model is then invariant along the subspace $\one$ and can be analyzed on
$\bm{1}_n^\perp$ only.

We will need to distinguish the even and odd parts of the couplings between agents, 
\begin{align}\label{eq:ks_dec}
 \dot{x}_i &= \omega_i - \sum_{j=1}^n c_{ij}\sin(x_i - x_j) - \sum_{j=1}^n s_{ij}\left[1 - \cos(x_i - x_j)\right] \notag\\ 
 &= \omega_i + \hod_i(\bm{x}) + \hev_i(\bm{x})\, ,
\end{align}
where we defined $c_{ij} = a_{ij}\cos(\varphi)$ and $s_{ij} = a_{ij}\sin\varphi$. 

A \emph{(frequency) synchronous state} is an $\bm{x}^{\rm s}\in\mathbb{R}^n$ such that
\begin{align}\label{eq:sync}
 \htot_i(\bm{x}^{\rm s}) &= \omega_i + \hod_i({\bm x}^{\rm s}) + \hev_i({\bm x}^{\rm s}) = \omega^{\rm s}\, ,
\end{align}
for some \emph{synchronous frequency} $\omega^{\rm s}\in\mathbb{R}$. 
Whereas a synchronous state is not necessarily an equilibrium of Eq.~\eqref{eq:ks}, it evolves along the invariant direction of the system, $\bm{x}^{\rm s} + \one$. 
Therefore, if for some time $t_0$, the system reaches a synchronous state $\bm{x}(t_0) = \bm{x}^{\rm s}$, then it remains synchronous for all time, i.e., for all $t>t_0$, 
\begin{align} \label{eq:sync-t}
 \bm{x}(t) &= \bm{x}^{\rm s} + (t-t_0) \omega^{\rm s} \bm{1}_n\, .
\end{align}

\subsection{\bf Winding cells and cohesive sets}\label{sec:winding}
Over the last decades, a variety of tools have been introduced to refine the analysis of Kuramoto oscillators networks. 
We present here some of these tools, that will be useful later. 
We refer to Ref.~\cite{Jaf22} for a comprehensive introduction to the following concepts. 

{\bf Counterclockwise difference.}
To account for the toroidal topology of the state space of Eq.~\eqref{eq:ks}, we defined the \emph{counterclockwise difference}~\cite{Jaf22}  between two numbers $x_1,x_2\in\mathbb{R}$,  
\begin{align}\label{eq:dcc}
 \dcc(x_1,x_2) &= x_1-x_2 + 2\pi k\, ,
\end{align}
where $k\in\mathbb{Z}$ is such that $\dcc(x_1,x_2)\in[-\pi,\pi)$. 
One notices that Eq.~\eqref{eq:ks} is invariant under substitution of $\dcc(x_1,x_2)$ for $x_1-x_2$.

{\bf Winding numbers and vectors.}
Given a cycle $\sigma=(i_0,i_1,...,i_\ell=i_0)$ of the interaction graph, we define the \emph{winding number} associated to $\bm{x}\in\mathbb{R}^n$, 
\begin{align}\label{eq:wn}
 q_\sigma(\bm{x}) &= \frac{1}{2\pi}\sum\nolimits_{j=1}^\ell \dcc(x_{i_{j-1}},x_{i_j}) \in \mathbb{Z}\, .
\end{align}
A collection of cycles $\Sigma=(\sigma_1,...,\sigma_c)$ of $\Graph$ induces the associated \emph{winding vector}, 
\begin{align}\label{eq:wv}
 q_\Sigma(\bm{x}) &= [q_{\sigma_1}(\bm{x}),...,q_{\sigma_c}(\bm{x})]^\top \in \mathbb{Z}^c\, .
\end{align}

{\bf Winding cells.}
Equation~\eqref{eq:wv} associates an element of the discrete space $\mathbb{Z}^c$ to any point $\bm{x}\in\mathbb{R}^n$. 
Reversing the map naturally defines a partition of $\mathbb{R}^n$ into \emph{winding cells}, 
\begin{align}\label{eq:wc}
 \Omega_\Sigma(\bm{u}) &= \left\{\bm{x}\in\mathbb{R}^n \colon  q_\Sigma(\bm{x}) = \bm{u} \right\}\, , & \bm{u}\in\mathbb{Z}^c\, .
\end{align}
It is shown in Ref.~\cite{Jaf22} that, considered as a partition of the $n$-torus, winding cells are connected subsets of $\mathbb{T}^n$. 
As we are working in $\mathbb{R}^n$, winding cells are periodic copies of (dynamically) equivalent connected sets (see Fig.~\ref{fig:flow}). 

\begin{figure}
 \centering
 
 \vspace{2mm}
 
 \includegraphics[width=.4\textwidth]{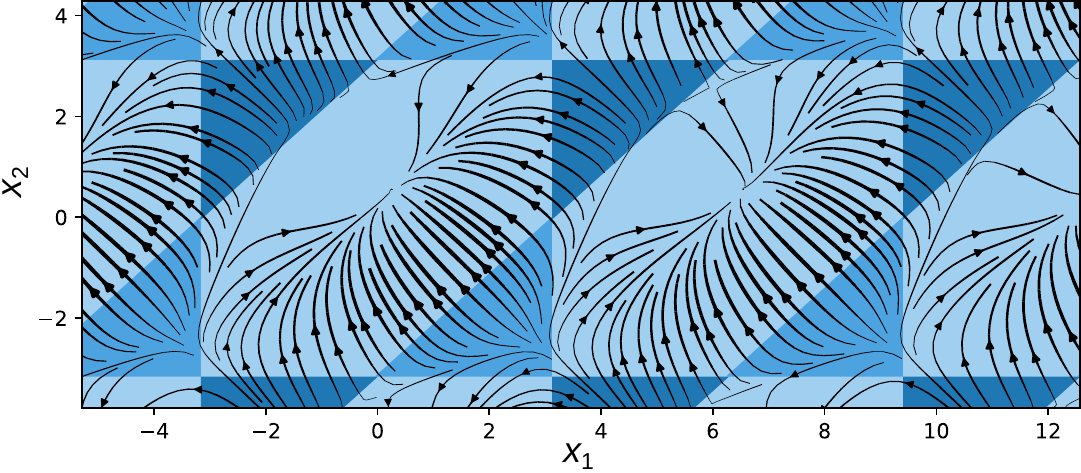}
 \caption{The Kuramoto-Sakaguchi vector field on the complete graph over 3
   nodes.  Only the first two components $(\dot{x}_1,\dot{x}_2)$ of the
   vector field are shown in the $(x_1,x_2,0)$ slice of $\mathbb{R}^3$.
   The background color illustrates the winding number [see
     Eq.~(\ref{eq:wn})] at the corresponding point of the state space
   (light blue: $q=0$, medium blue: $q=+1$, dark blue: $q=-1$).  The set of
   points sharing the same color then belong to the same winding cell.  One
   readily sees that both the vector field and the winding numbers are
   $2\pi$-periodic in all dimensions. We refer to~\cite{Jaf22,Del22} for
   3-dimensional visualizations of winding cells.}
 \label{fig:flow}
\end{figure}

{\bf Cohesive sets.}
For $\gamma\in[0,\pi]$, a state $\bm{x}\in\mathbb{R}^n$ is \emph{$\gamma$-cohesive} if for each edge $(i,j)$ of $\Graph$, $|\dcc(x_i,x_j)| \leq \gamma$. 
The \emph{$\gamma$-cohesive set}, $\cohes(\gamma)$, gathers all cohesive states, and the \emph{$\gamma$-cohesive $\bm{u}$-winding cell} is  
\begin{align}\label{eq:cohes-wc}
 \Dug &= \Omega_\Sigma(\bm{u}) \cap \cohes(\gamma)\, .
\end{align}
The set $\Dug$ will be our semicontractive domain for the Kuramoto-Sakaguchi model.

\section{Semicontraction theory}\label{sec:semic-recap}
We now review semicontraction theory, focusing on the notions needed to prove semicontractivity of the Kuramoto-Sakaguchi model.
We refer to~\cite[Chap.~5]{Bul23} for a comprehensive presentation and to~\cite{DeP23} for recent advances.

A \emph{seminorm} is a function $\|\cdot\| \colon \mathbb{R}^n \to \mathbb{R}$ satisfying 
\begin{align}\label{eq:semin-def}
 \|\alpha x\| &= |\alpha|\cdot \|x\| & \alpha&\in\mathbb{R}\, , \\
 \|x+y\| &\leq \|x\| + \|y\|\, .
\end{align}
Each seminorm has a \emph{kernel}, i.e., a vector subspace where the seminorm vanishes. 
Each seminorm is a norm when restricted to the subspace perpendicular to the kernel.

The Kuramoto-Sakaguchi model being invariant under homogeneous angle shifts, we consider seminorms whose kernel is the consensus space $\operatorname{span} (\bm{1}_n)$. 
Let $R\in\mathbb{R}^{(n-1)\times n}$ be a full rank matrix whose rows form an orthonormal basis of the subspace $\bm{1}_n^\perp$, and define the orthogonal projection
\begin{align}
 \Pin &= R^\top R = I_n - n^{-1}\bm{1}_n\bm{1}_n^\top\, .
\end{align}
In what follows we consider the \emph{$(\ell_2,\Pi_n)$ consensus seminorm} defined by $\|\bm{x}\|_\cons = \|\Pi_n\bm{x}\|_2=\|R\bm{x}\|_2$.
This choice of seminorm is somewhat arbitrary, but will be justified \emph{a posteriori} by the results it will allow. 

Each  (semi)norm $\|\cdot\|$ naturally induces a matrix (semi)norm, itself
inducing the matrix \emph{logarithmic (semi)norm} \cite[Def. 4]{Jaf22b}
\cite[Secs. 2.4 \& 5.3]{Bul23}
\begin{align}\label{eq:logseminorm}
 \mu(A) &\coloneqq \lim_{h\to 0^+}\frac{\|I_n + hA\| - 1}{h}\, .
\end{align}
We denote the logarithmic seminorm induced by the consensus seminorm as
$\mu_\cons$.  Logarithmic (semi)norms enjoy the following basic and well
known properties:
\begin{align}\label{eq:logsemi-prop}
 \mu(\alpha A) &= \alpha\mu(A) & \alpha &\geq 0\, , \\
 \mu(A+B) &\leq \mu(A) + \mu(B)\, , \label{eq:subadd}\\
 \mu(A) &\leq \|A\|\, .
\end{align}

When a seminorm is induced by a full-rank matrix $R$, the corresponding
logarithmic seminorm satisfies~\cite[Theorem 6]{Jaf22b}
\begin{align}\label{eq:n2sn}
 \mu_R(A) &= \mu(RAR^\dagger)\, ,
\end{align}
with $\dagger$ denoting the Moore-Penrose
pseudoinverse~\cite[Ex. 7.3.P7]{Hor94}.  In addition to
Eq.~\eqref{eq:n2sn}, here are two other ways of computing the consensus
logarithmic seminorm of a matrix:
\begin{align}
 \mu_\cons(A) &= \lambda_{\max}\left(R\frac{A+A^\top}{2}R^\dagger\right)\, , \label{eq:l2-log-seminorm1} \\
 \mu_\cons(A) &= \min\left\{b \colon  \Pin A + A^\top\Pin \preceq 2b\Pin\right\}\, , \label{eq:l2-log-seminorm2}
\end{align}
where $\lambda_{\max}$ denotes the largest eigenvalue of a matrix.
Equation~\eqref{eq:l2-log-seminorm1} follows from inserting Eq.~\eqref{eq:n2sn} and the property $R^\dagger=R^\top$ into the second line of \cite[Tab. 2.1]{Bul23}. 
Equation~\eqref{eq:l2-log-seminorm2} follows from \cite[Lem. 5.8]{Bul23}. 
A remarkable consequence of Eq.~\eqref{eq:l2-log-seminorm2} is that, when $L_{\Graph}$ is a weighted Laplacian matrix,
\begin{align}\label{eq:muL}
 \mu_\cons(-L_{\Graph}) &= -\lambda_2(L_{\Graph})\, .
\end{align}

We are now ready to define (semi)contraction. 
A dynamical system $\dot{x}=g(x)$ over a convex domain $C\in\mathbb{R}^n$ is said to be \emph{strongly infinitesimally (semi)contracting} \cite[Def. 12]{Jaf22b} \cite[Defs. 3.8 \& 5.9]{Bul23} if the logarithmic (semi)norm of its Jacobian is upper bounded by a negative constant everywhere in $C$, 
\begin{align}
 \mu(\JacNoSpace{g}(x)) &\leq -c\, , ~ \forall x\in C\, ,
\end{align}
and $c>0$ is the \emph{(semi)contraction rate}. 
For an illustration of a semicontractive system, we refer to \cite[Figs. 5.3 and 5.4]{Bul23}.

\section{Infinitesimal semicontraction in the phase-cohesive winding cells}\label{sec:semicontraction}
The rotational invariance of the Kuramoto-Sakaguchi model implies that the system cannot be semicontracting everywhere in $\mathbb{R}^n$. 
Our only hope is then to find subsets of $\mathbb{R}^n$ where the system is semicontracting. 

Let us define the \emph{odd} and \emph{even Jacobian matrices} 
\begin{align}
 \left[\Jac{\hod}(\bm{x})\right]_{ij} &= \left\{
 \begin{array}{ll}
   c_{ij}\cos(x_i  -  x_j)\, , & i\neq j\, , \\
   -\sum_{k=1}^n c_{ik}\cos(x_i  -  x_k)\, , & i = j\, ,
 \end{array}
 \right. \label{eq:jaco}
 \\[.2ex]
  \left[\Jac{\hev}(\bm{x})\right]_{ij} &= \left\{
 \begin{array}{ll}
   -s_{ij}\sin(x_i  -  x_j)\, , & i\neq j\, , \\
   \sum_{k=1}^n s_{ik}\sin(x_i  -  x_k)\, , & i = j\, ,
 \end{array}
 \right. \label{eq:jace}
\end{align}
and the Jacobian matrix of Eq.~\eqref{eq:ks}
\begin{align}
 \Jac{\htot} &= \Jac{\hod} + \Jac{\hev}\, .
\end{align} 
Note that all three Jacobians above have zero row sum, the same sparsity pattern as $L_{\Graph}$, and that $\Jac{\hod}$ is symmetric. 

By subadditivity of the logarithmic seminorms [Eq.~\eqref{eq:subadd}], we can analyze the odd and even parts of $\htot$ independently. 
The two following lemmas provide bounds on the logarithmic seminorm of the odd and even Jacobians respectively. 

\begin{lemma}[Log-seminorm of the odd Jacobian]\label{lem:jaco}
 For any $\gamma$-cohesive state $\bm{x}\in\cohes(\gamma)$, the logarithmic seminorm of the odd Jacobian (\ref{eq:jaco}) is bounded as
 \begin{align}\label{eq:bound_muf}
  \mu_\cons(\Jac{\hod}(\bm{x})) &\leq -\cos(\varphi)\cos(\gamma)\lambda_2(L_{\Graph})\, ,
 \end{align}  
 where $\lambda_2(L_{\Graph})$ is the algebraic connectivity of the graph $\Graph$. 
\end{lemma}

\begin{proof}
  We notice that, under our assumptions, $-\Jac{\hod}$ is a symmetric
  Laplacian matrix with positive edge weights.  Eq.~\eqref{eq:muL} then
  implies
  \begin{align}
    \mu_\cons(\Jac{\hod}(\bm{x})) &= \lambda_2(\Jac{\hod})\, .
  \end{align}
  
Now, let us take $\bm{x}\in\cohes(\gamma)$, for $\gamma\in[0,\pi/2]$. 
Then, each off-diagonal term ($i\neq j$) of the Jacobian $\Jac{\hod}(\bm{x})$ satisfies
\begin{align}\label{eq:offdiag}
 0 &\leq a_{ij}\cos(\varphi) \cos(\gamma) \leq \left(\Jac{\hod}(\bm{x})\right)_{ij} \leq a_{ij} \cos(\varphi)\, ,
\end{align}
with the $a_{ij}$'s being the elements of the adjacency matrix of $\Graph$. 
Let us define the symmetric matrix 
\begin{align}
 L(\bm{x}) &= \Jac{\hod}(\bm{x}) + \cos(\varphi)\cos(\gamma) L_{\Graph}\, .
\end{align}
By Eq.~\eqref{eq:offdiag}, $L(\bm{x})$ has nonnegative off-diagonal elements, zero row sum, and therefore nonpositive diagonal elements (it is a Laplacian matrix). 
Hence, by Gershgorin Circles Theorem~\cite[Theorem 6.1.1]{Hor94}, $L(\bm{x})$ is negative semidefinite. 
Using a corollary of Weyl's inequality \cite[Cor. 4.3.12]{Hor94}, 
\begin{align}
 -\lambda_2(\Jac{\hod}(\bm{x})) &= \lambda_2\left(\cos(\varphi)\cos(\gamma) L_{\Graph} - L(\bm{x})\right) \notag\\
 &\geq \cos(\varphi)\cos(\gamma)\lambda_2(L_{\Graph}) 
\end{align}
which concludes the proof.
\end{proof}

\begin{lemma}[Log-seminorm of the even Jacobian]\label{lem:jace}
 For any $\gamma$-cohesive point $\bm{x}\in\cohes(\gamma)$, the logarithmic seminorm of the even Jacobian (\ref{eq:jace}) is bounded as 
 \begin{align}\label{eq:bound_mug}
  \mu_\cons(\Jac{\hev}(\bm{x})) &\leq \sin(\varphi)\sin(\gamma)\dmax\, .
 \end{align}
 where $\dmax$ denotes the maximal weighted degree of $\Graph$. 
\end{lemma}

\begin{proof}
We compute the logarithmic seminorm of $\Jac{\hev}$ using Eq.~\eqref{eq:l2-log-seminorm1}, the fact that $R^\dagger=R^\top$, and the skew-symmetry of $\Jac{\hev}$,  
\begin{align}
 \mu_\cons(\Jac{\hev}(\bm{x})) 
 &= \lambda_{\max}\left(R\frac{\Jac{\hev} + (\Jac{\hev})^\top}{2}R^\top\right) \notag\\
 &= \lambda_{\max}\left(R [{\rm diag}(\Jac{\hev})] R^\top\right)\, ,
\end{align}
where ${\rm diag}(A)\in\mathbb{R}^n$ is the diagonal of $A$, and $[\bm{v}]\in\mathbb{R}^{n\times n}$ is the diagonal matrix constructed with $\bm{v}\in\mathbb{R}^n$. 
Properties of interlacing eigenvalues \cite[Cor. 4.3.37]{Hor94} imply
\begin{align}
 \mu_\cons(\Jac{\hev}) &\leq \max({\rm diag}(\Jac{\hev}))\, . \label{eq:bound_even}
\end{align}

Under the assumption that $\bm{x}\in\cohes(\gamma)$ with $\gamma\in[0,\pi/2]$, we have 
\begin{align}
 |\sin(x_i-x_j)| &\leq \sin(\gamma)\, ,
\end{align}
for each connected pair $(i,j)$. 
Therefore, 
\begin{align}
 [\Jac{\hev}(\bm{x})]_{ii} &= \sum_{j=1}^n a_{ij}\sin(\varphi)\sin(x_i-x_j) \notag\\
 &\leq \sin(\varphi)\sin(\gamma)\dg_i\, ,
\end{align}
concluding the proof. 
\end{proof}

Combining Lemmas \ref{lem:jaco} and \ref{lem:jace} yields the first main theorem of this manuscript, showing local infinitesimal semicontractivity of the Kuramoto-Sakaguchi model. 

\begin{theorem}[Local strong semicontraction]\label{thm:max_cohes}
 Consider the Kuramoto-Sakaguchi model on a weighted, connected, undirected
 graph $\Graph$, with frustration parameter $\varphi\in[0,\pi/2]$.  Let
 $\lambda_2$ and $\dmax$ denote the algebraic connectivity and maximal
 degree of $\Graph$ respectively, and define the angle
 $\bar{\gamma}\in(0,\pi/2)$ by
 \begin{align}\label{eq:max_cohes}
  \bar{\gamma} &= \arctan\left[\frac{\lambda_2}{\dmax \tan(\varphi)}\right]\, .
 \end{align}
 Then for any $0<\gamma<\bar{\gamma}$, the Kuramoto-Sakaguchi model is
 strongly infinitesimally semicontracting in the $\gamma$-cohesive set
 $\cohes(\gamma)$ with respect to the $(\ell_2,\Pi_n)$ consensus seminorm.
\end{theorem}

\begin{proof}
  By assumption, there exists $0 < c'' < \bar{\gamma}-\gamma$.  Therefore,
  by definition of $\bar{\gamma}$,
  \begin{align}
    \frac{\sin(\gamma)}{\cos(\gamma)} &< \frac{\lambda_2}{\dmax\tan(\varphi)} - c'\, ,
  \end{align}
  for some $c'>0$.  Reorganizing the terms yields
  \begin{align}
    -\cos(\varphi)\cos(\gamma)\lambda_2 &< -\sin(\varphi)\sin(\gamma)\dmax - c\, , 
  \end{align}
  for an appropriate $c>0$.  We then conclude using
  Eqs.~\eqref{eq:bound_muf} and \eqref{eq:bound_mug}, and subadditivity of
  logarithmic seminorms
  \begin{align}
    \mu_\cons(\Jac{\htot}) &\leq \mu_\cons(\Jac{\hod}) + \mu_\cons(\Jac{\hev}) < - c< 0\, . 
  \end{align}
  where we used that $\bm{x}\in\cohes(\gamma)$. 
\end{proof}

Note that the contraction rate $c$ can be computed from the proof and from
a continuity argument to satisfy:
\begin{align*}
 c &= \cos(\varphi)\cos(\gamma)\tan(\bar{\gamma}-\gamma)
 \frac{\dmax^2\tan^2(\varphi)-\lambda_2}{\dmax\tan(\varphi) -
   \lambda_2\tan(\bar{\gamma}-\gamma)}\, .
\end{align*}

Theorem~\ref{thm:max_cohes} relates the local semicontractivity of the Kuramoto-Sakaguchi model, to the phase cohesiveness. 
Furthermore, the cohesiveness bound Eq.~\eqref{eq:max_cohes} depends on system parameters only, namely, on the graph structure through $\lambda_2$ and $\dmax$, and on the frustration $\varphi$. 
Eq.~\eqref{eq:max_cohes} clearly emphasizes how the ratio between algebraic connectivity and degree works in favor of synchrony, whereas frustration jeopardizes synchronization. 

\begin{figure}
 \centering
 
 \vspace{2mm}
 
 \includegraphics[width=.4\textwidth]{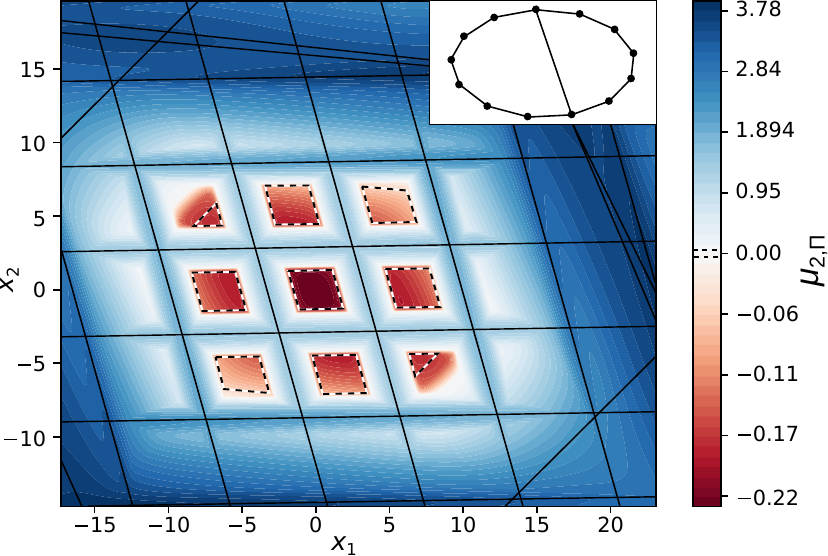}
 \caption{Two-dimensional slice of $\mathbb{R}^{13}$, showing the Jacobian's logarithmic seminorm for the Kuramoto-Sakaguchi model on $\Graph$ (see inset) and frustration $\varphi=0.01$. 
 The plain black lines show the boundaries of the winding cells. 
 The dashed black lines are the boundaries of the $\bar{\gamma}$-cohesive $\bm{u}$-winding cells $\Gamma_{\bm{u},\bar{\gamma}}$, for $\bm{u}\in\{-1,0,1\}^2$. 
 Theorem~\ref{thm:max_cohes} guarantees that the system is semicontracting within these cohesive winding cells as can be confirmed by the value of the logarithmic seminorm therein. 
 Notice the change in color scale between the positive and negative values of the logarithmic seminorm.
 }
 \label{fig:heatmap}
\end{figure}

In Fig.~\ref{fig:heatmap}, we illustrate to what extent the bound in Eq.~\eqref{eq:max_cohes} is conservative. 
The proximity of the boundary of the cohesive winding cell (dashed black line) and the area where the logarithmic seminorm turns positive (blue) confirms that our bound Eq.~\eqref{eq:max_cohes} is rather sharp. 
In other numerical experiments, we discovered that the condition in Theorem~\ref{thm:max_cohes} is more conservative for denser network structures. 
We also illustrate this observation in Fig.~\ref{fig:bounds}, where we show the dependence of $\bar{\gamma}$ on both the frustration and the ratio between algebraic connectivity and maximal degree.

\section{At most uniqueness of synchronous states}\label{sec:amu}
Contracting systems over convex sets are known to have at most a unique equilibrium \cite[Ex. 3.16]{Bul23}. 
Here we extend this property to the Kuramoto-Sakaguchi model by showing that the set over which it is semicontracting is actually convex. 

\begin{figure}
 \centering
 
 \vspace{3mm}
 
 \includegraphics[width=.38\textwidth]{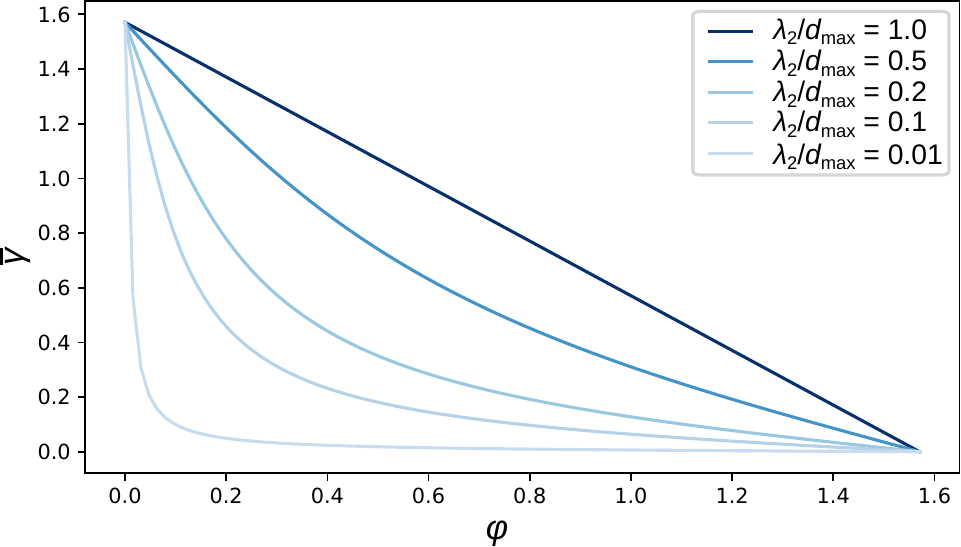}
 \caption{Illustration of the upper bound $\bar{\gamma}$ of
   Theorem~\ref{thm:max_cohes} as a function of the frustration $\varphi$,
   for various values of the ratio $\lambda_2/\dmax$.  High values of
   $\lambda_2/\dmax$ typically correspond to densely connected graphs,
   confirming the intuition that such graphs tolerate less cohesive stable
   synchronous states, and vice-versa for low ratio $\lambda_2/\dmax$. }
 \label{fig:bounds}
\end{figure}

\begin{theorem}\label{thm:amu}
 Let $0<\gamma<\bar{\gamma}$, with $\bar{\gamma}$ defined in Eq.~(\ref{eq:max_cohes}), and let $\bm{u}\in\mathbb{Z}^c$. 
 Up to a constant phase shift and up to the addition of integer multiples of $2\pi$, there is at most one synchronous state of the Kuramoto-Sakaguchi model in the $\gamma$-cohesive $\bm{u}$-winding cell $\Dug$.
\end{theorem}

\begin{proof}[Sketch, details in Appendix~\ref{sec:thm4}]
The proof of Theorem~\ref{thm:amu} leverages a natural polytopic representation (see Appendix~\ref{sec:polytope}) of each connected component of the $\gamma$-cohesive winding cell, presented in Ref.~\cite{Jaf22}. 
The polytopic representation of the winding cells conveniently allows the application of standard results of semicontraction theory, guaranteeing exponential convergence of the system. 
In summary, infinitesimal semicontractivity implies strong infinitesimal contractivity of a reduced system and the polytopic representation of the state space guarantees convexity of the contractive subdomain, implying \emph{at most uniqueness} of the equilibrium therein. 
\end{proof}

Theorem~\ref{thm:amu} improves previous \emph{at most uniqueness} conditions \cite[Theorem 4.1]{Jaf22} and \cite[Corollary 6]{Del22} in that it applies to a wider class of models, has a less conservative condition on cohesiveness [Eq.~\eqref{eq:max_cohes}], and the condition for semicontractivity [Eq.~\eqref{eq:max_cohes}] is much more explicit than its counterpart in \cite[Eq. (45)]{Del22}. 
Furthermore, the proof is much more streamlined and elegant, which provides a clear insight on the underlying mechanisms leading to \emph{at most uniqueness}, namely the local semicontractivity of the Kuramoto-Sakaguchi model. 
In Refs.~\cite{Jaf22,Del22}, the proof methods are {\it ad hoc} and tailored to the specific problem at hand.

\section{Conclusion}
Networks of Kuramoto-Sakaguchi oscillators gather all subtleties that need to be considered when applying semicontraction theory to a dynamical system. 
Due to rotational invariance, the Kuramoto-Sakaguchi model is not contracting, but at best semicontracting. 
Therefore, we first formulated the system in an appropriate framework, where semicontraction theory applies, and we defined the appropriate semicontraction metric, namely, the consensus seminorm. 
Furthermore, systems of Kuramoto-Sakaguchi oscillators are not semicontracting on the whole state space. 
Hence, we had to identify convex subdomains where semicontraction holds, which are the phase cohesive winding cells. 

Tackling the question of synchronization in networks of Kuramoto-Sakaguchi
oscillators through the prism of (semi)contraction shed light on the
mechanisms leading to multistability.  Moreover, contracting and
semicontracting systems come with various convenient properties.  In
particular, resorting to semicontraction theory in our proof implies that
it is robust against disturbances, noise, or unmodeled dynamics.

Also, contracting systems naturally come with associated Lyapunov functions \cite[Eq. (3.28)]{Bul23}. 
Knowing these Lyapunov functions for the Kuramoto-Sakaguchi model opens new paths of investigation for open questions about networks of phase oscillators, such as the computation of attraction regions of synchronous states or the design of algorithms for the computation of fixed points.

\appendix

\subsection{Polytopic representation of the state space} \label{sec:polytope}
It is shown in Ref.~\cite[Theorem 3.6]{Jaf22} that for any $\bm{x}\in\Omega_\Sigma(\bm{u})$, there is a unique $\bm{y}\in\bm{1}_n^\perp$ such that
\begin{align}
 \dcc(B^\top\bm{x}) &= B^\top\bm{y} + 2\pi C^\dagger_\Sigma \bm{u}\, ,
\end{align}
where $B$ is the incidence matrix of the graph $\Graph$ and $C_\Sigma$ is the cycle-edge incidence matrix, defined, e.g., in \cite[Eq. (2.1)]{Jaf22}.
Notice that even though the proof of \cite[Theorem 3.6]{Jaf22} is provided for points on the $n$-torus, the same proof directly extends to $\mathbb{R}^n$ through the natural construction of the $n$-torus as a quotient space of the Euclidean space. 

The point $\bm{y}\in\bm{1}_n^\perp$ is easily reparametrized as $\bm{z}=R\bm{y}\in\mathbb{R}^{n-1}$, and one notices that $\bm{y}=R^\dagger\bm{z}$. 
By construction, $\bm{z}$ belongs to the polytope
\begin{align}\label{eq:polytope}
 \Pu &= \left\{\bm{z}\in\mathbb{R}^{n-1} \colon  \|B^\top R^\dagger\bm{z} + 2\pi C_\Sigma^\dagger\bm{u}\|_\infty \leq \pi\right\}\, ,
\end{align}
and we therefore have a natural projection 
\begin{align}\label{eq:proj}
\begin{split}
 \proj \colon \Omega_\Sigma(\bm{u}) &\to \Pu\, .
\end{split}
\end{align}
Furthermore, Ref.~\cite[Theorem 3.6]{Jaf22} shows in particular that the projection is injective. 
Restricting the projection $\proj$ to phase cohesive sets, one verifies that the $\gamma$-cohesive $\bm{u}$-winding cell, $\Dug$, is injectively mapped to the \emph{$\gamma$-cohesive $\bm{u}$-polytope}
\begin{align}
 \Pug &= \left\{\bm{z}\in\mathbb{R}^{n-1} \colon  \|B^\top R^\dagger\bm{z} + 2\pi C_\Sigma^\dagger\bm{u}\|_\infty < \gamma\right\}\, ,
\end{align}
which is convex by construction.

Let us see how the dynamics induced by the Kuramoto-Sakaguchi model [Eq.~\eqref{eq:ks}] on $\bm{x}$ translates to the evolution of its image $\proj(\bm{x})=\bm{z}\in\Pu$, 
\begin{align}
 \dot{\bm z} &= R(B^\top)^\dagger B^\top R^\dagger\dot{\bm z} 
 = R(B^\top)^\dagger \frac{d}{dt}(B^\top R^\dagger\bm{z}) \notag\\
 &= R(B^\top)^\dagger \frac{d}{dt}\left(\dcc(B^\top\bm{x}) - 2\pi C_\Sigma^\dagger\bm{u}\right)\, .
\end{align}
Now, by definition of the counterclockwise difference, there exists $\bm{k}\in\mathbb{Z}^m$ such that $\dcc(B^\top\bm{x}) = B^\top\bm{x} + 2\pi\bm{k}$. 
Furthermore, for almost all $\bm{x}$, the value of $\bm{k}$ is constant in a neighborhood of $B^\top\bm{x}$. 
Therefore, almost everywhere,
\begin{align}\label{eq:ydot2}
 \dot{\bm z} &= R(B^\top)^\dagger \frac{d}{dt}\left(B^\top\bm{x} + 2\pi\bm{k} - 2\pi C_\Sigma^\dagger\bm{u}\right) \notag\\
 &= R(B^\top)^\dagger B^\top \dot{\bm x} 
 = R\dot{\bm x} = R\htot(\bm{x})\, ,
\end{align}
where we used that $\bm{u}$ is constant.
A posteriori, we verify that the above derivation does not depend on the choice of preimage $\bm{x}$ of $\bm{z}$.
Indeed, two points $\bm{x}_1, \bm{x}_2 \in \Omega_\Sigma(\bm{u})$ that have the same image $\bm{z}\in\Pu$ satisfy $\htot(\bm{x}_1) = \htot(\bm{x}_2)$, and Eq.~\eqref{eq:ydot2} is valid everywhere. 

The diffusive nature of the coupling functions $\htot_i$ implies that one can re-write them as functions of the pairwise component differences along the edges of the interaction graph. 
Namely, one can define $\Htot_i\colon\mathbb{R}^m\to\mathbb{R}$ such that $\htot_i(\bm{x}) = \Htot_i(B^\top\bm{x})$. 
We then get 
\begin{align}\label{eq:ydot}
 \dot{\bm z} &= R \htot(\bm{x}) 
 = R \Htot(B^\top\bm{x})
 = R \Htot(\dcc(B^\top\bm{x})) \notag\\
 &= R \Htot(B^\top R^\dagger\bm{z} + 2\pi C_\Sigma^\dagger\bm{u}) 
 = \tilde{\htot}(\bm{z})\, .
\end{align}

\subsection{Proof the \emph{at most uniqueness} (Theorem~\ref{thm:amu})} \label{sec:thm4}
We first need to prove that the system Eq.~\eqref{eq:ydot} is strongly infinitesimally contracting. 

\begin{lemma}\label{lem:amu}
 Let $\bar{\gamma}$ as in Eq.~(\ref{eq:max_cohes}), $0<\gamma<\bar{\gamma}$, and $\bm{u}\in\mathbb{Z}^c$. 
 Then Eq.~\eqref{eq:ydot} has at most a unique fixed point in $\Pug$. 
\end{lemma}

\begin{proof}
Using the chain rule and the property of pseudoinverses that $AA^\dagger A = A$, we compute the Jacobian of the system, 
\begin{align}\label{eq:indep}
 \Jac{\tilde{\htot}}(\bm{z}) &= R~ \Jac{\Htot}(B^\top R^\dagger\bm{z} + 2\pi C_\Sigma^\dagger\bm{u}) B^\top R^\dagger \notag\\
 &= R~ \Jac{\Htot}(\dcc(B^\top \bm{x})) B^\top (B^\top)^\dagger B^\top R^\dagger \notag\\
 &= R~ \Jac{\Htot}(B^\top\bm{x}) B^\top R^\dagger 
 = R \Jac{\htot}(\bm{x}) R^\dagger\, ,
\end{align}
independently of the choice for $\bm{x}$ the preimage of $\bm{z}$.
Equation~\eqref{eq:indep} allows us to compute the logarithmic norm of the system $\dot{\bm z}=\tilde{\htot}(\bm{z})$, 
\begin{align}\label{eq:id_logseminorm}
 \mu_2(\Jac{\tilde{\htot}}(\bm{z})) &= \mu_2(R\Jac{\htot}(\bm{x})R^\dagger) = \mu_\cons(\Jac{\htot}(\bm{x}))\, ,
\end{align}
where we used Eq.~\eqref{eq:n2sn}, independently of the preimage $\bm{x}$. 

For any point $\bm{z}\in\Pug$, any of its preimage $\bm{x}$ is $\gamma$-cohesive by definition. 
Therefore, by Eq.~\eqref{eq:id_logseminorm} and Theorem~\ref{thm:max_cohes}, the logarithmic norm of the system Eq.~\eqref{eq:ydot} is strictly negative and the system is strongly infinitesimally contracting on $\Pug$. 

Furthermore, by definition, the phase cohesive polytope $\Pug$ is convex. 
In summary, the system $\dot{\bm z}=\tilde{\htot}(\bm{z})$ is strongly infinitesimally contracting on a convex set. 
Note that $\Pug$ is not necessarily forward invariant. 
Therefore, by \cite[Ex. 3.16]{Bul23}, there is at most a unique fixed point in $\Pug$.
\end{proof}

We are now ready to prove Theorem~\ref{thm:amu}. 
Let $\xs\in\Dug$ be a synchronous state of Eq.~\eqref{eq:ks}, i.e., $\htot(\xs) = \omega^{\rm s}\bm{1}_n$. 
One can verify that $\bm{z}^{\rm s}=\proj(\bm{x}^{\rm s})\in\Pug$ is a fixed point of Eq.~\eqref{eq:ydot}, i.e., $\tilde{\htot}(\bm{z}^{\rm s})=0$. 
Therefore, if $\bm{x}^{\rm t}\in\Dug$ is also a synchronous state of Eq.~\eqref{eq:ks}, then $\proj(\xt)\in\Pug$ is also a fixed point of Eq.~\eqref{eq:ydot} and Lemma~\ref{lem:amu} directly implies that $\proj(\xt)=\bm{z}^{\rm s}$. 

The projection $\proj$ is such that $\dcc(B^\top\xs) = \dcc(B^\top\xt)$, i.e., $B^\top(\xs-\xt) = 2\pi\bm{k}$, for some $\bm{k}\in\mathbb{Z}^m$. 
The graph $\Graph$ being connected, we can recursively reconstruct the vector of differences $\xs-\xt$ up to a constant shift. 
We conclude that for each node $i$, the difference $(\xs-\xt)_i$ is of the form $(\xs-\xt)_i = \rho + 2\pi\ell_i$, with $\ell_i\in\mathbb{Z}$, where we do not know $\rho\in\mathbb{R}$, but it does not matter. 
Up to a constant shift $\rho$ of all components and up to an integer multiple of $2\pi$, $\xs$ and $\xt$ are then the same synchronous state, which concludes the proof.

\end{document}